\documentclass[11pt,letterpaper,reqno]{amsart}

\usepackage{amsmath,amsthm,amsfonts,amssymb,amscd}
\usepackage{empheq}
\usepackage{lipsum}
\usepackage{lastpage}
\usepackage{thmtools}
\usepackage{enumerate}
\usepackage[shortlabels]{enumitem}
\usepackage{fancyhdr}
\usepackage{mathrsfs}
\usepackage{xcolor}
\usepackage{graphicx}
\usepackage{listings}
\usepackage{bbm}
\usepackage{array}
\usepackage{tabularx}
\usepackage{hyperref}
\usepackage[makeroom]{cancel}

\usepackage[utf8]{inputenc}
\usepackage{comment}

\usepackage{makecell}

\usepackage[all]{xy}
\xyoption{matrix}
\xyoption{arrow}

\mathtoolsset{showonlyrefs=true}
\numberwithin{figure}{section}
\numberwithin{table}{section}

\theoremstyle{definition}

\theoremstyle{plain}
\newcommand{\thistheoremname}{}
\newtheorem*{genericthm*}{\thistheoremname}
\newenvironment{namedthm*}[1]
  {\renewcommand{\thistheoremname}{#1}%
   \begin{genericthm*}}
  {\end{genericthm*}}

\hypersetup{%
  colorlinks=true,
  linkcolor=blue,
  citecolor=blue,
  linkbordercolor={0 0 1}
}

\newcommand{\lrabs}[1]{\left\lvert #1 \right\lvert}
\newcommand{\lrp}[1]{\left(#1\right)}
\newcommand{\lrb}[1]{\left[#1\right]}

\allowdisplaybreaks

\newtheorem{theorem}{Theorem}[section]
\newtheorem{corollary}[theorem]{Corollary}
\newtheorem{lemma}[theorem]{Lemma}

\newtheorem{mytable}[theorem]{Table}

\theoremstyle{remark}
\newtheorem{conjecture}[theorem]{\bf Conjecture}

\numberwithin{equation}{section}

\usepackage{latexsym}
\usepackage{amssymb}

\newcommand{\ben}{\begin{equation}}
\newcommand{\een}{\end{equation}}

\makeatletter
\NewDocumentCommand{\sump}{e{_}}
 {%
  \DOTSB
  \mathop{\IfNoValueTF{#1}{\sump@{}}{\sump@{#1}}}%
  \nolimits
 }
\newcommand{\sump@}[1]{\mathpalette\sump@@{#1}}
\newcommand{\sump@@}[2]{%
  \ifx#1\displaystyle
    {\sump@display{#2}}%
  \else
    \sum@\nolimits'_{#2}%
  \fi
}
\newcommand{\sump@display}[1]{%
  \sbox\z@{$\m@th\displaystyle\sum@\nolimits'$}%
  \sbox\tw@{$\m@th\displaystyle\sum@\limits_{#1}$}%
  \sbox\@tempboxa{$\m@th\displaystyle'$}
  \mathop{\sum@\nolimits' \kern-\wd\@tempboxa}\limits_{#1}%
  \ifdim\wd\z@>\wd\tw@
    \kern\dimexpr\wd\z@-\wd\tw@\relax
  \fi
}
\makeatother

\newcommand{\Av}{{\textnormal{A\!v}}}

\DeclareMathOperator{\Tr}{Tr}

\DeclareMathOperator*{\Res}{Res}

\setlist[enumerate]{leftmargin=*,widest=0}
\setlist[itemize]{leftmargin=*,widest=0}
\makeatletter
\def\subsection{\@startsection{subsection}{2}%
  \z@{.5\linespacing\@plus.7\linespacing}{.3\linespacing}%
  {\normalfont\bfseries}}

\def\subsubsection{\@startsection{subsubsection}{3}%
  \z@{.5\linespacing\@plus.7\linespacing}{.3\linespacing}%
  {\normalfont\bfseries}}
\makeatother


\author[W. Cason]{William Cason}
\address[W. Cason]{School of Mathematical and Statistical Sciences, Clemson University, Clemson, SC, 29634}
\email{wbcason@clemson.edu}

\author[A. Jim]{Akash Jim}
\address[A. Jim]{Department of Mathematics, Princeton University, Princeton, NJ, 08540}
\email{ajim@princeton.edu}

\author[C. Medlock]{Charlie Medlock}
\address[C. Medlock]{School of Mathematical and Statistical Sciences, Clemson University, Clemson, SC, 29634}
\email{medloc5@clemson.edu}

\author[E. Ross]{Erick Ross}
\address[E. Ross]{School of Mathematical and Statistical Sciences, Clemson University, Clemson, SC, 29634}
\email{erickr@clemson.edu}

\author[H. Xue]{Hui Xue}
\address[H. Xue]{School of Mathematical and Statistical Sciences, Clemson University, Clemson, SC, 29634}
\email{huixue@clemson.edu}

\title{On the average size of the eigenvalues of the Hecke operators}

\begin{document}


\begin{abstract}
We determine the average size of the eigenvalues of the Hecke operators acting on the cuspidal modular forms space $S_k(\Gamma_0(N))$ in both the vertical and the  horizontal perspective. The ``average size" is measured via the quadratic mean. 
\end{abstract}

\maketitle

\section{Introduction}\label{sec:intro}

Let $m\ge1$, $N\ge 1$, and $k \ge 2$ be positive integers, and let $T'_m(N,k) := T_m(N,k) / m^{(k-1)/2}$ denote the $m$-th normalized Hecke operator over $S_k(\Gamma_0(N))$. When $N$ and $k$ are clear from the context, we will just write $T'_m$.
In this paper, we study the average size of the eigenvalues for these Hecke operators.

One could interpret the notion of average size here in two different ways: vertically and horizontally. 
In Section \ref{sec:vertical}, we consider the vertical perspective; $m$ is fixed and $N$ and $k$ vary. For $m$ fixed, consider $N$ coprime to $m$ and $k \geq 2$ even. Also, let $s(N,k) = \dim S_k(\Gamma_0(N))$ and  $\lambda_1,\cdots, \lambda_{s(N,k)}$ be the eigenvalues of $T'_m(N,k)$ (with multiplicities). Then to measure the average size of the eigenvalues of $T'_m(N,k)$, we use the quadratic mean,
\begin{align} \label{eqn:Av-m-def}
    \Av_m(N,k):=\sqrt{\frac{1}{s(N,k)}\sum_{i=1}^{s(N,k)} \lambda_i^2}.
\end{align}
We note that since $N$ is coprime to $m$, the $\lambda_i$ are real \cite[Corollary 10.3.7]{cohen-stromberg}, and hence the above formula actually represents the average size. 
We also note that in this situation, the quadratic mean is more natural to study than, for example, the arithmetic mean of the absolute values of the $\lambda_i$; squares are much nicer to work with than absolute values. 
In Theorem \ref{thm:Av-vertical}, we will determine the behavior of $\Av_m(N,k)$ as $N+k \to \infty$. In the following, $\sigma_1(m)$ denotes the sum of divisors function, $\sigma_1(m) = \sum_{d\mid m} d$.
\begin{theorem} \label{thm:Av-vertical}
Let $m$ be fixed. Then for $N$ coprime to $m$ and $k \geq 2$ even,
$$\Av_m(N,k) \longrightarrow \sqrt{\frac{\sigma_1(m)}{m}} \quad \textnormal{as} \quad N+k \longrightarrow \infty.$$
\end{theorem}
Our method is effective; we have explicit bounds determining exactly how fast $\Av_m(N,k)$ converges to $\sqrt{\frac{\sigma_1(m)}{m}}$.
As an example, in Section \ref{sec:computations}, we use these explicit bounds to 
classify precisely when $\Av_2(N,k) \le 1$. 

In Section \ref{sec:horizontal}, we consider the horizontal perspective; $N$ and $k$ are fixed and $m$ varies. From this perspective, it is more natural to think of the eigenvalues of $T'_m(N,k)$ as the normalized Fourier coefficients of newforms. Let $f$ be a newform of weight $k$ and level $N$. Let $a'_f(m) := a_f(m) / m^{(k-1)/2}$ denote the normalized Fourier coefficients of $f$, i.e. the eigenvalues of $T'_m(N,k)$ associated to $f$. Then to measure the average size of the $a'_f(m)$, we again use the quadratic mean,
\begin{align} \label{eqn:Avf-def}
    \Av_f(x) := \sqrt{ \frac{1}{x} \sum_{m \leq x} a'_f(m)^2 }. 
\end{align}
We note that since $f$ is a newform, the $a'_f(m)$ are real \cite[Remark 13.3.12]{cohen-stromberg}, and hence the above formula actually represents the average size.
In Theorem \ref{thm:Av-horizontal}, we determine the behavior of $\Av_f(x)$ as $x \to \infty$. 
\begin{theorem} \label{thm:Av-horizontal}
Let $f$ be a newform of level $N$ and weight $k$. Then
$$\Av_f(x) \longrightarrow \sqrt{\frac{12 \cdot (4\pi)^{k-1}}{(k-1)!}} \cdot \Vert f \Vert \quad \textnormal{as} \quad x \longrightarrow \infty.$$
Here, $\Vert f \Vert$ denotes the Petersson norm, $\Vert f \Vert = \sqrt{\langle f,f \rangle}$.
\end{theorem}

The idea of this theorem is not new. Although perhaps not stated in these terms, others have considered equivalent questions (e.g. \cite{rankin}, \cite{selberg}, \cite{ivic-et-al}). However, we state this theorem here for completeness.

Finally, in Section \ref{sec:discussion}, we discuss our results and propose a lower bound on the size of the Fourier coefficients of newforms, which mimics the Atkin-Serre conjecture.

\section{The Vertical Perspective} \label{sec:vertical}
In this section, we take the vertical perspective. That is, for $m$ fixed, we study $\Av_m(N,k)$ as $N+k \to \infty$. Recall that here we are considering $N$ coprime to $m$ and $k \ge 2$ even. Note that $s(N, k) = 0$ for only finitely many pairs $(N,k)$ \cite[Table 2.6]{ross}, so $\Av_m(N,k)$ here is well-defined for $N + k$ sufficiently large.

Now, the case of $m=p$ prime follows from already known results. In \cite[Equation 27]{serre}, Serre showed that assuming $p \nmid N$, the asymptotic distribution of the eigenvalues of $T'_p(N,k)$ is given by $\displaystyle \mu_p = \frac{p+1}{\pi} \frac{(1-t^2/4)^{1/2} }{ (p^{1/2} + p^{-1/2})^2 - t^2 }  \,dt$.
This means that 
\begin{align}
    \lim_{N+k \to \infty} \Av_p(N,k)^2 
    &=   \lim_{N+k \to \infty} \frac{1}{s(N, k)} \displaystyle \sum_{i=1}^{s(N,k)} \lambda_i^2  \\
    &=   \int_{-2}^2 \frac{p+1}{\pi} \frac{(1-t^2/4)^{1/2} }{ (p^{1/2} + p^{-1/2})^2 - t^2 }  \cdot t^2 \,dt \\
    &=  \frac{p+1}{p}, 
\end{align}
and hence
\begin{equation}
    \Av_p(N,k) \longrightarrow \sqrt{\frac{p+1}{p}} \qquad \textnormal{as} \qquad N+k \longrightarrow \infty. \label{eqn:Av-asymp-prime}
\end{equation}

However, this approach will not work for general $m$ because Serre's methods do not produce a measure $\mu_m$ for composite $m$. Instead, we compute the asymptotics of $\Av_m(N,k)$ directly by writing $\Av_m(N,k)$ in terms of traces of certain Hecke operators. 

Let $\lambda_1, \ldots, \lambda_{s(N,k)}$ be the eigenvalues of $T'_m$. Then using the Hecke operator composition formula \cite[Theorem 10.2.9]{cohen-stromberg}, 
\begin{align}
    \sum_{i=1}^{s(N,k)} \lambda_{i}^2 
    &= \Tr \lrp{{T'_m}^2} = \Tr \sum_{d\mid m}  T'_{m^2/d^2} = \sum_{d\mid m}  \Tr T'_{m^2/d^2}.
\end{align}
Thus by \eqref{eqn:Av-m-def}, 
\begin{equation} \label{eqn:Av-m-trace-formula}
    \Av_m(N,k) = \sqrt{\frac{1}{s(N, k)}  \sum_{d\mid m} \Tr T'_{m^2/d^2}}.
\end{equation}

We now cite a lemma to estimate the traces of Hecke operators appearing in this formula. 
In the following, all big-$O$ notation is vertical, i.e. with respect to $N$ and $k$. 

\begin{lemma}[{\cite[Lemma 4.2]{ross-xue}, \cite[Proposition 4]{serre}}] \label{lem:Tr-asymptotics} 
Let $m$ be a fixed perfect square. Then for $N$ coprime to $m$ and $k \geq 2$ even,
\begin{align} \label{eqn:trace-estimate-perfect-square}
\Tr T'_m = \frac{k-1}{12} \psi(N) \frac{1}{\sqrt{m}} + O(N^{1/2+\varepsilon}) \qquad \text{for any } \varepsilon > 0.
\end{align}
\end{lemma}
Here, $\psi(N)$ is the multiplicative function $\psi(N) = N \prod_{p\mid N} (1 + 1/p)$. Note that $\psi(N) \geq N$. We then use this lemma to determine the asymptotic behavior of $\Av_m(N,k)$.

{
    \renewcommand{\thetheorem}{\ref{thm:Av-vertical}}
    \begin{theorem}
    Let $m$ be fixed. Then for $N$ coprime to $m$ and $k \geq 2$ even,
$$\Av_m(N,k) \longrightarrow \sqrt{\frac{\sigma_1(m)}{m}} \quad \textnormal{as} \quad N+k \longrightarrow \infty.$$
    \end{theorem}
    \addtocounter{theorem}{-1}
}
\begin{proof}

By \eqref{eqn:Av-m-trace-formula} and Lemma \ref{lem:Tr-asymptotics},
\begin{align}
    \Av_m(N,k)^2 &= \frac{1}{s(N, k)} \sum_{d\mid m} \Tr T'_{m^2/d^2}  \\ 
    &= \frac{1}{s(N, k)} \sum_{d\mid m}  \lrp{ \frac{k-1}{12} \psi(N)  \frac{d}{m} + O(N^{1/2+\varepsilon}) } \\
    &= \frac{\sigma_1(m)}{m} \cdot  \frac{\frac{k-1}{12} \psi(N)}{s(N, k)}  + \frac{O(N^{1/2+\varepsilon})}{s(N, k)}. \label{eqn:Av-m-asymp-temp}
\end{align}

Note from Lemma \ref{lem:Tr-asymptotics} that $s(N, k) = \Tr T'_1 = \frac{k-1}{12} \psi(N) + O(N^{1/2+\varepsilon})$.  
Thus as $N+k \longrightarrow \infty$, 
\begin{align}
    \frac{\frac{k-1}{12} \psi(N)}{s(N, k)}  \longrightarrow 1 \qquad\qquad \text{and} \qquad\qquad \frac{O(N^{1/2+\varepsilon})}{s(N, k)} \longrightarrow 0. 
\end{align}

So by \eqref{eqn:Av-m-asymp-temp}, $\Av_m(N,k)^2 \longrightarrow \frac{\sigma_1(m)}{m}$, proving the desired result. 
\end{proof}

\section{A classification of when \texorpdfstring{$\Av_2(N,k) \le 1$}{Av\_2(N,k)} }\label{sec:computations}

Our method is effective; we have explicit bounds for all of the error terms in the proof of Theorem \ref{thm:Av-vertical}, which allows us to determine exactly how fast $\Av_m(N,k)$ converges to $\sqrt{\frac{\sigma_1(m)}{m}}$. 
The fact that $\Av_m(N,k) \longrightarrow \sqrt{\frac{\sigma_1(m)}{m}}$ means that for any $\alpha < \sqrt{\frac{\sigma_1(m)}{m}}$, we will have $\Av_m(N,k) \leq \alpha$ for only finitely many pairs $(N,k)$. Recall that here, we are considering $N$ coprime to $m$ and $k\ge 2$ even. We take $m=2,\, \alpha=1$ in particular, and compute the complete list of pairs $(N,k)$ for which $\Av_2(N,k) \leq 1$.

\begin{theorem} \label{thm:Av-2-computation}
Consider $N\geq1$ odd and $k\geq2$ even such that $\dim S_k(\Gamma_0(N)) > 0$. Then $\Av_2(N,k) \leq 1$ only for the pairs $(N,k)$ given in Table \ref{table:Av-le-1}.
\end{theorem}

\begin{proof}
From \eqref{eqn:Av-m-trace-formula},
\begin{align}
    \Av_2(N,k) &= \sqrt{1 + \frac{\Tr T_{4}'}{ s(N,k)}}.
\end{align}
So $\Av_2(N,k) \le 1$ precisely when $\Tr T_4' \le 0$. By Lemma \ref{lem:Tr-asymptotics} we have that $\Tr T'_4 = \frac{k-1}{24} \psi(N) + O(N^{1/2+\varepsilon})$ for any $\varepsilon>0$. And in fact, the proof of \cite[Lemma 6.1]{ross-xue} gave the explicit lower bound 
\begin{align} \label{eqn:T4-lower-bound-temp}
    \Tr T_4'(N,k) &\geq  \psi(N) \lrb{\frac{k-1}{24} - \lrp{14 \cdot \frac{2^{\omega(N)}}{\psi(N)} + \frac12 \frac{2^{\omega(N)} \sqrt{N}}{\psi(N)}} }.
\end{align}
Then for $N \geq 150,\!000$, \cite[Lemma 3.1]{ross-xue} gives the explicit bounds 
\begin{align} \label{eqn:theta-i-explicit-bounds}
    \frac{2^{\omega(N)}}{\psi(N)} \leq 0.000147 \qquad \text{and} \qquad \frac{2^{\omega(N)} \sqrt{N}}{\psi(N)} \leq 0.0607, 
\end{align}
which means that
\begin{align}
    \lrp{14 \cdot \frac{2^{\omega(N)}}{\psi(N)} + \frac12 \frac{2^{\omega(N)} \sqrt{N}}{\psi(N)}} \leq 14 \cdot 0.000147 + \frac12 \cdot 0.0607 < \frac{1}{24} \leq \frac{k-1}{24}.
\end{align}
This means that by \eqref{eqn:T4-lower-bound-temp}, $\Tr T'_4(N,k) > 0$ for all $N \geq 150,\!000$. Then for each of the finitely many remaining values of $N$, we will also have $\frac{k-1}{24} > \lrp{14 \frac{2^{\omega(N)}}{\psi(N)} + \frac12 \frac{2^{\omega(N)} \sqrt{N}}{\psi(N)}}$, and hence $\Tr T'_4(N,k) > 0$, for sufficiently large $k$.
Finally, we check the finitely many remaining pairs $(N,k)$ by computer, which yields the complete list given in Table \ref{table:Av-le-1}. See \cite{ross-code} for the code.
\end{proof}

\begin{mytable} \label{table:Av-le-1}
\begin{equation}
\setlength{\arraycolsep}{1.7mm}
\begin{array}{|c|c||c|c||c|c||c|c|}
\hline
\multicolumn{8}{|c|}{\makecell{
\text{The complete list of pairs $(N,k)$ for which $\Av_2(N,k) \leq 1$.  } 
}} \\
\hline
 (N,k) & \Av_2  &  (N,k) & \Av_2  &  (N,k) & \Av_2  &  (N,k) & \Av_2  \\
\hline
\hline
(1, 12) & \frac{3}{8}\sqrt{2} & (5, 6) & \frac14\sqrt{2} & (21, 2) & \frac{1}{2}\sqrt{2} & (53, 2) & 1 \\
\hline
(1, 20) & \frac{57}{128} \sqrt{2} & (7, 4) & \frac14\sqrt{2} & (23, 2) & \frac12\sqrt{3} & (57, 2) & \frac{3}{10}\sqrt{10} \\
\hline
(1, 22) & \frac{9}{64}\sqrt{2}  & (9, 4) & 0 & (27, 2) & 0 & (61, 2) & 1 \\
\hline
(1, 26) & \frac{3}{512}\sqrt{2} & (11, 4) & \frac{1}{2}\sqrt{2} & (31, 2) & \frac12\sqrt{3} & (63, 2) & \frac{3}{10}\sqrt{10} \\
\hline
(1, 30) & \frac{3}{512} \sqrt{26687} & (15, 2) & \frac{1}{2}\sqrt{2} & (37, 2) & 1 & (81, 2) & \frac{1}{2}\sqrt{3} \\
\hline
(1, 46) & \frac{3}{2^{16}}\sqrt{468559893} & (17, 2) & \frac{1}{2}\sqrt{2} & (45, 2) & \frac{1}{2}\sqrt{2}  & \multicolumn{2}{c}{~} \\     
\cline{1-6}
(3, 8) & \frac{3}{8}\sqrt{2} & (19, 2) & 0 & (49, 2) & \frac{1}{2}\sqrt{2}  & \multicolumn{2}{c}{~} \\
\cline{1-6}
\end{array}
\end{equation}
\end{mytable}

\section{The Horizontal Perspective}\label{sec:horizontal} 

In this section, we take the horizontal perspective. That is, for $f$ fixed, we study $\Av_f(x)$ as $x \to \infty$, as defined in \eqref{eqn:Avf-def}. Recall that $f$ here is a newform of weight $k$ and level $N$, and that $a'_f(m)$ denotes the $m$-th normalized Fourier coefficient of $f$.

We note that, similarly to the vertical perspective, the average size of the \textit{prime-indexed} normalized Fourier coefficients follows from already known results.
When $f$ is CM, the distribution of the $a_f'(p)$ tends to the measure $\mu_{CM} = \lrp{\frac{1}{2}\delta_0(t) + \frac{1}{4\pi}\frac{1}{\sqrt{1 - t^2/4}}} dt$, where $\delta_0$ is the Dirac delta function \cite{arias-de-reyna-et-al}, \cite[Theorem 15.4]{harmonic-maass-forms}. Consequently,
\begin{align} \label{eqn:avg-afp2-cm}
    \lim_{x \to \infty} \frac{1}{\#\{ \text{primes } p \leq x\} } \sum_{  p \leq x }  a'_f(p)^2  
    &= \int_{-2}^2 \lrp{\frac{1}{2}\delta_0(t) + \frac{1}{4\pi}\frac{1}{\sqrt{1 - t^2/4}} } \cdot t^2 \,dt = 1. 
\end{align}
When $f$ is non-CM, Barnet-Lamb, Geraghty, Harris, and Taylor showed that the distribution of the $a'_f(p)$ tends to the Sato-Tate measure $\displaystyle \mu_{ST} = \frac{1}{\pi} \sqrt{1-t^2/4}  \,dt$ as $p \to \infty$ \cite{BGHT}, \cite{arias-de-reyna-et-al}. Consequently,   
\begin{align} \label{eqn:avg-afp2}
    \lim_{x \to \infty} \frac{1}{\#\{ \text{primes } p \leq x\} } \sum_{  p \leq x }  a'_f(p)^2  
    &= \int_{-2}^2 \frac{1}{\pi} \sqrt{1-t^2/4} \cdot t^2 \,dt = 1.
\end{align}
In this sense, the average size of the $a'_f(p)$ is $\sqrt{1} = 1$, regardless of whether $f$ is CM or not. Interestingly, if we consider the general $r$-th mean (as opposed to just using the quadratic mean at $r=2$), then
$$\lim_{x \to \infty} \lrp{ \frac{1}{\#\{ \text{primes } p \leq x\} } \sum_{  p \leq x }  \lrabs{a'_f(p)}^r}^{1/r}$$
agrees for CM and non-CM newforms only at $r = 2$. One can verify this fact by computing the two integrals directly. 

The average size of \textit{all} the normalized Fourier coefficients of $f$ also follows from already known results. Recall that we are interested in determining
\begin{align} 
    \lim_{x\to\infty} \Av_f(x) =  \sqrt{ \lim_{x \to \infty} \frac{1}{x}  \sum_{  m \leq x }  a'_f(m)^2}  .
\end{align}

This problem reduces to determining the asymptotic behavior of \\ ${\sum_{  m \leq x }  a'_f(m)^2}$. By the Wiener-Ikehara theorem \cite[Theorem 1.2]{murty-et-al}, this in turn reduces to determining the residue of the Dirichlet series $L_f(s) := \sum_{m = 1}^\infty a'_f(m)^2 m^{-s}$ at $s=1$. This Dirichlet series is closely related to the symmetric square $L$-function, and has already been studied extensively; in particular, its residues are known.
From \cite[Corollary 11.12.3]{cohen-stromberg}, we have that $L_f(s)$ can be extended meromorphically to the entire complex plane, and that it has a simple pole at $s=1$ of residue $\Res_{s=1} L_f(s) = \frac{12 \cdot (4\pi)^{k-1}}{(k-1)!} \langle f,f \rangle$. This fact was originally proven by Rankin \cite[Theorem 3]{rankin} in order to show that 
\begin{align} \label{eqn:avg-tau-unnorm}
\sum_{m\leq x} \tau(m)^2 \ \sim \ \lrp{\frac{1}{12}  \Res_{s=1} L_\Delta(s)} \cdot x^{12} = \frac{(4\pi)^{11}}{11!} \langle \Delta,\Delta \rangle \cdot x^{12}
\end{align}
In our case, we are interested in something slightly different, the \textit{normalized} Fourier coefficients. By the Wiener-Ikehara Theorem, we have $\sum_{m\leq x} a'_f(m)^2 \ \sim \ \lrp{\Res_{s=1} \, L_f(s)} \cdot x$. This means that \begin{align}
    \lim_{x\to\infty} \Av_f(x)^2 = \lim_{x\to\infty} \frac{1}{x} \sum_{m\leq x} a'_f(m)^2 = \Res_{s=1} \, L_f(s) = \frac{12 \cdot (4\pi)^{k-1}}{(k-1)! } \langle f,f \rangle,
\end{align}  
verifying Theorem \ref{thm:Av-horizontal}.

As an example, we use this theorem to give the average size of the normalized Ramanujan $\tau$ function. Note the distinction between the average size and the constant given in \eqref{eqn:avg-tau-unnorm} by Rankin, which is due to our normalization of the Fourier coefficients.

\begin{corollary} The average size of the normalized Ramanujan $\tau$ function is given by   
\begin{align}
    \lim_{x\to\infty} \Av_\Delta(x) = \sqrt{\frac{12 \cdot (4\pi)^{11}}{(11)!} } \Vert \Delta \Vert = 0.619745...
\end{align}
\end{corollary}
\noindent For this numerical value, we used $\Vert \Delta \Vert = 0.001017527... $ from \cite[p. 286]{cohen-stromberg}.

\section{Discussion}\label{sec:discussion}

For any newform $f$, recall that the optimal
upper bound on the size of $a'_f(m)$ (for composite $m$) is a natural generalization of the optimal upper bound on the size of $a'_f(p)$ (for prime $p$). We have Deligne's bound $\lrabs{a'_f(p)} \leq 2$, and the natural generalization $\lrabs{a'_f(m)} \leq \sigma_0(m)$. 
Additionally, we have from \eqref{eqn:avg-afp2-cm}, \eqref{eqn:avg-afp2}, and Theorem \ref{thm:Av-horizontal} that the average size of the $a'_f(m)$ is a (nonzero) constant factor of the average size of the $a'_f(p)$.

For these reasons, we also conjecture that the Fourier coefficients of $f$ should have a lower bound that resembles a lower bound on the prime-indexed Fourier coefficients. In particular, we propose Conjecture \ref{conj:atkin-serre-gen}, which mimics the Atkin-Serre conjecture \cite[p. 2]{rouse2007atkin} on prime-indexed Fourier coefficients. 
\begin{conjecture} \label{conj:atkin-serre-gen}
    Let $f$ be a 
    newform of weight $k \geq 4$ and level $N$. 
    Then for all $\varepsilon > 0$, there exists a constant $c_\varepsilon > 0$ such that for all $m\ge1$, if
    $a_f(m) \neq 0$, then $\lrabs{a_f(m)} \geq c_\varepsilon m^{(k-3)/2 - \varepsilon}$.
\end{conjecture}
This conjecture aligns with our numerical computations \cite{ross-code}.
We note that contrary to the classical Atkin-Serre conjecture, we must account for the fact that it is possible to have $a_f(m) = 0$ for infinitely many $m$. By multiplicativity, even if $a_f(p) = 0$ for only one prime $p$, then $a_f(m) = 0$ for infinitely many $m$. Restricting to $m$ such that $a_f(m) \neq 0$ also makes it possible to remove the ``non-CM" condition from classical Atkin-Serre.

\section*{Acknowledgements}
This research was supported by NSA MSP grant H98230-24-1-0033. We are grateful for the referee's careful reading of this paper and valuable advice.

\bibliographystyle{plain}
\bibliography{bibliography.bib}

\begin{thebibliography}{10}

\bibitem{arias-de-reyna-et-al}
Sara Arias-de Reyna, Ilker Inam, and Gabor Wiese.
\newblock On conjectures of {S}ato-{T}ate and {B}ruinier-{K}ohnen.
\newblock {\em Ramanujan J.}, 36(3):455--481, 2015.

\bibitem{BGHT}
Tom Barnet-Lamb, David Geraghty, Michael Harris, and Richard Taylor.
\newblock A family of {C}alabi-{Y}au varieties and potential automorphy {II}.
\newblock {\em Publ. Res. Inst. Math. Sci.}, 47(1):29--98, 2011.

\bibitem{harmonic-maass-forms}
Kathrin Bringmann, Amanda Folsom, Ken Ono, and Larry Rolen.
\newblock {\em Harmonic {M}aass forms and mock modular forms: theory and applications}, volume~64 of {\em American Mathematical Society Colloquium Publications}.
\newblock American Mathematical Society, Providence, RI, 2017.

\bibitem{cohen-stromberg}
Henri Cohen and Fredrik Str\"{o}mberg.
\newblock {\em Modular forms: A classical approach}, volume 179 of {\em Graduate Studies in Mathematics}.
\newblock American Mathematical Society, Providence, RI, 2017.

\bibitem{ivic-et-al}
Aleksandar Ivi\'c, Kohji Matsumoto, and Yoshio Tanigawa.
\newblock On {R}iesz means of the coefficients of the {R}ankin-{S}elberg series.
\newblock {\em Math. Proc. Cambridge Philos. Soc.}, 127(1):117--131, 1999.

\bibitem{murty-et-al}
M.~Ram Murty, Jagannath Sahoo, and Akshaa Vatwani.
\newblock A simple proof of the {W}iener-{I}kehara {T}auberian theorem.
\newblock {\em Expo. Math.}, 42(3):Paper No. 125570, 11, 2024.

\bibitem{rankin}
R.~A. Rankin.
\newblock Contributions to the theory of {R}amanujan's function {$\tau(n)$} and similar arithmetical functions. {II}. {T}he order of the {F}ourier coefficients of integral modular forms.
\newblock {\em Proc. Cambridge Philos. Soc.}, 35:351--372, 1939.

\bibitem{ross}
Erick Ross.
\newblock Newspaces with nebentypus: An explicit dimension formula, classification of trivial newspaces, and character equidistribution property.
\newblock Submitted 2024. \url{https://arxiv.org/abs/2407.08881}.

\bibitem{ross-code}
Erick Ross.
\newblock Average size of {H}ecke eigenvalues.
\newblock ~\\ \url{https://github.com/eross156/average-size-hecke-eigenvalues}, 2024.

\bibitem{ross-xue}
Erick Ross and Hui Xue.
\newblock Signs of the second coefficients of {H}ecke polynomials.
\newblock Submitted 2024. \url{https://arxiv.org/abs/2407.10951}.

\bibitem{rouse2007atkin}
Jeremy Rouse.
\newblock Atkin-{S}erre type conjectures for automorphic representations on {${\rm GL}(2)$}.
\newblock {\em Math. Res. Lett.}, 14(2):189--204, 2007.

\bibitem{selberg}
Atle Selberg.
\newblock Bemerkungen \"uber eine {D}irichletsche {R}eihe, die mit der {T}heorie der {M}odulformen nahe verbunden ist.
\newblock {\em Arch. Math. Naturvid.}, 43:47--50, 1940.

\bibitem{serre}
Jean-Pierre Serre.
\newblock R\'{e}partition asymptotique des valeurs propres de l'op\'{e}rateur de {H}ecke {$T_p$}.
\newblock {\em J. Amer. Math. Soc.}, 10(1):75--102, 1997.

\end{thebibliography}

\end{document}